\documentclass[11pt,twoside]{article}
\usepackage{amsmath, amsthm, amscd, amsfonts, amssymb, graphicx, color}

\date{}

\begin{document}

\centerline{}

\centerline {\Large{\bf  Fusion frame and its alternative dual in tensor product }}

\centerline{\Large{\bf  of Hilbert spaces}}

\newcommand{\mvec}[1]{\mbox{\bfseries\itshape #1}}
\centerline{}
\centerline{\textbf{Prasenjit Ghosh}}
\centerline{Department of Pure Mathematics, University of Calcutta,}
\centerline{35, Ballygunge Circular Road, Kolkata, 700019, West Bengal, India}
\centerline{e-mail: prasenjitpuremath@gmail.com}
\centerline{}
\centerline{\textbf{T. K. Samanta}}
\centerline{Department of Mathematics, Uluberia College,}
\centerline{Uluberia, Howrah, 711315,  West Bengal, India}
\centerline{e-mail: mumpu$_{-}$tapas5@yahoo.co.in}

\newtheorem{Theorem}{\quad Theorem}[section]

\newtheorem{definition}[Theorem]{\quad Definition}

\newtheorem{theorem}[Theorem]{\quad Theorem}

\newtheorem{remark}[Theorem]{\quad Remark}

\newtheorem{corollary}[Theorem]{\quad Corollary}

\newtheorem{note}[Theorem]{\quad Note}

\newtheorem{lemma}[Theorem]{\quad Lemma}

\newtheorem{example}[Theorem]{\quad Example}

\newtheorem{result}[Theorem]{\quad Result}
\newtheorem{conclusion}[Theorem]{\quad Conclusion}

\newtheorem{proposition}[Theorem]{\quad Proposition}

\begin{abstract}
\textbf{\emph{We study fusion frame in tensor product of Hilbert spaces and discuss some of its properties.\,The resolution of the identity operator on a tensor product of Hilbert spaces is being discussed.\,An alternative dual of a fusion frame in tensor product of Hilbert spaces is being presented.}}
\end{abstract}
{\bf Keywords:}  \emph{Fusion frame, Resolution of identity operator, Canonical dual \\ \smallskip\hspace{2.2 cm}frame, Tensor product of Hilbert spaces, Tensor product of frames.}

{\bf 2010 Mathematics Subject Classification:} \emph{42C15; 46C07.}

\section{Introduction}

Fusion frame was introduced by P.\,Casazza and G.\,Kutyniok \cite{Kutyniok}.\,They define frames for closed subspaces of a given Hilbert spaces with respect to the orthogonal projections.\;Fusion frame is a natural generalization of the frame theory in Hilbert space and it has so many applications in data processing, coding theory, signal processing and many other fields.\,Infact, the fusion frame theory is more elegent due to complicated relations between the structure of the sequence of weighted subspace and the local frames in the subspace.\;In fusion frame, the atomic resolution of the identity operator on Hilbert space was studied by M.\,S Asgari and Amil Khosraki \cite{Asgari}.\,They present a reconstruction formula and establish some useful results about resolution of the identity operator.

The basic concepts of tensor product of Hilbert spaces were described by S.\,Rabin-son \cite{S}.\;Frames and Bases in Tensor Product of Hilbert spaces were introduced by A.\,Khosravi and M.\,S.\,Asgari \cite{A}.\;Reddy et al.\,\cite{Upender} also studied the frame in tensor product of Hilbert spaces and presented the frame operator on tensor product of Hilbert spaces.\;The concepts of fusion frames and \,$g$-frames in tensor product of Hilbert spaces were introduced by Amir Khosravi and M.\,Mirzaee Azandaryani \cite{Mir}.

In this paper, fusion frame in tensor product of Hilbert spaces is developed and discuss a relationship among fusion frames in Hilbert spaces and their tensor products.\,We shall verify that in tensor product of Hilbert spaces, an image of a fusion frame under a bounded linear operator will be a fusion frame if the operator is invertible and unitary.\;The resolution of the identity operator on a tensor product of Hilbert spaces is presented.\,In tensor product of Hilbert spaces, we study an alternative dual of a fusion frame and see that the canonical dual of a fusion frame is also an alternative dual.\,Finally, we establish that an alternative dual of a fusion frame is a fusion frame in tensor product of Hilbert spaces. 

Throughout this paper,\;$H \;\text{and}\; K$\, are considered to be separable Hilbert spaces with associated inner products \,$\left <\,\cdot \,,\, \cdot\,\right>_{1} \;\text{and}\; \left <\,\cdot \,,\, \cdot\,\right>_{2}$, respectively.\,$I_{H}\, \,\text{and}\, \,I_{K}$\, are the identity operators on \,$H\, \,\text{and}\, \,K$, respectively.\;$\mathcal{B}\,(\,H \,,\, K\,)$\; is the collection of all bounded linear operators from \,$H \;\text{to}\; K$.\;In particular, \,$\mathcal{B}\,(\,H\,)$\, denote the space of all bounded linear operators on \,$H$.\;$P_{\,V}$\, denote the orthogonal projection onto the closed subspace \,$V \,\subset\, H$.\;$\left\{\,V_{i}\,\right\}_{ i \,\in\, I}$\, and \,$\left\{\,W_{j}\,\right\}_{ j \,\in\, J}$\, are the collections of closed subspaces of \,$H \;\text{and}\; K$, where \,$I,\, J$\, are index sets.\;Define the space
\[l^{\,2}\left(\,\left\{\,V_{i}\,\right\}_{ i \,\in\, I}\,\right) \,=\, \left \{\,\{\,f_{\,i}\,\}_{i \,\in\, I} \,:\, f_{\,i} \,\in\, V_{i} \,,\, \sum\limits_{\,i \,\in\, I}\, \left \|\,f_{\,i}\,\right \|_{1}^{\,2} \,<\, \infty \,\right\}\]
with inner product is given by \,$\left<\,\{\,f_{\,i}\,\}_{ i \,\in\, I} \,,\, \{\,g_{\,i}\,\}_{ i \,\in\, I}\,\right> \,=\, \sum\limits_{\,i \,\in\, I}\, \left<\,f_{\,i} \,,\, g_{\,i}\,\right>_{1}$.\;Clearly \,$l^{\,2}\left(\,\left\{\,V_{i}\,\right\}_{ i \,\in\, I}\,\right)$\; is a Hilbert space with respect to the above inner product \cite{Asgari}.\;Similarly, we can define the space \,$l^{\,2}\left(\,\left\{\,W_{j}\,\right\}_{j \,\in\, J}\,\right)$.

\section{Preliminaries}

\smallskip\hspace{.6 cm}

\begin{theorem}\cite{Gavruta}\label{th0.001}
Let \,$V \,\subset\, H$\; be a closed subspace and \,$T \,\in\, \mathcal{B}\,(\,H\,)$.\;Then \,$P_{\,V}\, T^{\,\ast} \,=\, P_{\,V}\,T^{\,\ast}\, P_{\,\overline{T\,V}}$.\;If \,$T$\, is an unitary operator (\,i\,.\,e \,$T^{\,\ast}\, T \,=\, I_{H}$\,), then \,$P_{\,\overline{T\,V}}\;T \,=\, T\,P_{\,V}$.
\end{theorem}

\begin{theorem}\cite{Jain}\label{th0.01}
The set \,$\mathcal{S}\,(\,H\,)$\; of all self-adjoint operators on \,$H$\; is a partially ordered set with respect to the partial order \,$\leq$\, which is defined as for 
\[T,\,S \,\in\, \mathcal{S}\,(\,H\,), \,T \,\leq\, S \,\Leftrightarrow\, \left<\,T\,(\,f\,) \,,\, f\,\right>_{1} \,\leq\, \left<\,S\,(\,f\,) \,,\, f\,\right>_{1}\; \;\forall\; f \,\in\, H.\] 
\end{theorem}

\begin{definition}\cite{Kutyniok}
Let \,$\left\{\,v_{i}\,\right\}_{ i \,\in\, I}$\, be a collection of positive weights.\;A family of weighted closed subspaces \,$V \,=\, \left\{\, (\,V_{i},\, v_{i}\,) \,:\, i \,\in\, I\,\right\}$\, is called a fusion frame for \,$H$\, if there exist constants \,$0 \,<\, A \,\leq\, B \,<\, \infty$\, such that
\begin{equation}\label{eq0.01}
A \;\left\|\,f \,\right\|_{1}^{\,2} \,\leq\, \sum\limits_{\,i \,\in\, I}\, v_{i}^{\,2} \;\left\|\, P_{\,V_{i}}\,(\,f\,) \,\right\|_{1}^{\,2} \,\leq\, B \, \left\|\, f \, \right\|_{1}^{\,2}\; \;\forall\; f \,\in\, H.
\end{equation}
The constants \,$A,\, B$\; are called fusion frame bounds.\;If the family \,$V$\, satisfies 
\[\sum\limits_{\,i \,\in\, I}\, v_{i}^{\,2} \;\left\|\, P_{\,V_{i}}\,(\,f\,) \,\right\|_{1}^{\,2} \,\leq\, B \, \left\|\, f \, \right\|_{1}^{\,2}\; \;\forall\; f \,\in\, H\] then it is called a fusion Bessel sequence in \,$H$\, with bound \,$B$.
\end{definition}

\begin{definition}\cite{Kutyniok}\label{def1}
Let \,$V \,=\, \left\{\,\left(\,V_{i},\, v_{i}\,\right)\,\right\}_{i \,\in\, I}$\, be a fusion Bessel sequence in \,$H$\, with bound \,$B$.\;The synthesis operator \,$T_{V} \,:\, l^{\,2}\left(\,\left\{\,V_{i}\,\right\}_{ i \,\in\, I}\,\right) \,\to\, H$\, is defined as 
\[T_{V}\,\left(\,\left\{\,f_{\,i}\,\right\}_{i \,\in\, I}\,\right) \,=\,  \sum\limits_{\,i \,\in\, I}\, v_{i}\,f_{i}\; \;\;\forall\; \{\,f_{i}\,\}_{i \,\in\, I} \,\in\, l^{\,2}\left(\,\left\{\,V_{i}\,\right\}_{ i \,\in\, I}\,\right)\] and the analysis operator is given by 
\[ T_{V}^{\,\ast} \,:\, H \,\to\, l^{\,2}\left(\,\left\{\,V_{i}\,\right\}_{ i \,\in\, I}\,\right),\; T_{V}^{\,\ast}\,(\,f\,) \,=\,  \left\{\,v_{i}\, P_{\,V_{i}}\,(\,f\,)\,\right\}_{ i \,\in\, I}\; \;\forall\; f \,\in\, H.\]
The fusion frame operator \,$S_{V} \,:\, H \,\to\, H$\; is defined as follows:
\[S_{V}\,(\,f\,) \,=\, T_{V}\,T_{V}^{\,\ast}\,(\,f\,) \,=\, \sum\limits_{\,i \,\in\, I}\, v_{i}^{\,2}\; P_{\,V_{i}}\,(\,f\,)\; \;\forall\; f \,\in\, H.\]
\end{definition}

\begin{note}\cite{Kutyniok}
Let \,$V$\, be a fusion frame with bounds \,$A,\,B$.\;Then from (\ref{eq0.01}), 
\[\left<\,A\,f \,,\, f\,\right>_{1} \,\leq\, \left<\,S_{V}\,(\,f\,) \,,\, f\,\right>_{1} \,\leq\, \left<\,B\,f \,,\, f\,\right>_{1}\; \;\forall\; f \,\in\, H.\]
The operator \,$S_{V}$\, is bounded, self-adjoint, positive and invertible.\;Now, according to the Theorem (\ref{th0.01}), we can write, \,$A\,I_{H} \,\leq\,S_{V} \,\leq\, B\,I_{H}$\, and this gives \,$B^{\,-1}\,I_{H} \,\leq\, S_{V}^{\,-1} \,\leq\, A^{\,-1}\,I_{H}$.
\end{note}

\begin{definition}\cite{Kutyniok}
Let \,$V \,=\, \left\{\,\left(\,V_{i},\, v_{i}\,\right)\,\right\}_{i \,\in\, I}$\, be a fusion frame for \,$H$.\,Then \,$\left\{\,\left(\,S_{V}^{\,-1}\,V_{i},\, v_{i}\,\right)\,\right\}_{i \,\in\, I}$\, is called the canonical dual fusion frame of \,$V$. 
\end{definition}

\begin{theorem}\cite{Gavruta}\label{thm1.1}
Let \,$V \,=\, \left\{\,\left(\,V_{i},\, v_{i}\,\right)\,\right\}_{i \,\in\, I}$\, be a fusion frame for \,$H$\, with bounds \,$A,\, B$\, and \,$S_{V}$\, be the corresponding frame operator.\;Then the canonical dual fusion frame of \,$V$\, is a fusion frame with bounds \,$\dfrac{A}{\left\|\,S_{V}\,\right\|^{\,2}\,\left\|\,S^{\,-\, 1}_{V}\,\right\|^{\,2}}\,,\;B\,\left\|\,S_{V}\,\right\|^{\,2}\,\left\|\,S^{\,-\, 1}_{V}\,\right\|^{\,2}$. 
\end{theorem}

\begin{note}\cite{Gavruta}\label{note1}
A reconstruction formula on \,$H$\, with the help of canonical dual fusion frame is given by
\[f \,=\, \sum\limits_{\,i \,\in\, I}\,v^{\,2}_{\,i}\,P_{\,S^{\,-\, 1}_{V}\,V_{\,i}}\,S^{\,-\, 1}_{V}\,P_{\,V_{\,i}}\,(\,f\,)\;\; \;\forall\; f \,\in\, H.\]
\end{note}

\begin{definition}\cite{Gavruta}
Let \,$V \,=\, \left\{\,\left(\,V_{i},\, v_{i}\,\right)\,\right\}_{i \,\in\, I}$\, be a fusion frame for \,$H$\, and \,$S_{V}$\, be the corresponding frame operator.\;Then a fusion Bessel sequence \,$V^{\,\prime} \,=\, \left\{\,\left(\,V^{\,\prime}_{i},\, v^{\,\prime}_{i}\,\right)\,\right\}_{i \,\in\, I}$\, is said to be an alternative dual of \,$V$\, if
\[f \,=\, \sum\limits_{\,i \,\in\, I}\,v_{\,i}\,v^{\,\prime}_{\,i}\,P_{\,V^{\,\prime}_{\,i}}\,S^{\,-\, 1}_{V}\,P_{\,V_{\,i}}\,(\,f\,)\;\; \;\forall\; f \,\in\, H.\] 
\end{definition}

\begin{definition}\cite{Kutyniok}
A family of bounded operators \,$\left\{\,T_{i}\,\right\}_{i \,\in\, I}$\, on \,$H$\, is called a resolution of identity operator on \,$H$\; if for all \,$f \,\in\, H$, we have \,$f \,=\, \sum\limits_{\,i \,\in\, I}\,T_{i}\,(\,f\,)$, provided the series converges unconditionally for all \,$f \,\in\, H$.
\end{definition}

There are several ways to introduced the tensor product of Hilbert spaces.\;The tensor product of Hilbert spaces \,$H$\, and \,$K$\, is a certain linear space of operators which was represented by Folland in \cite{Folland}, Kadison and Ringrose in \cite{Kadison}.\\

\begin{definition}\cite{Upender}
The tensor product of \,$H$\, and \,$K$\, is denoted by \,$H \,\otimes\, K$\, and it is defined to be an inner product space associated with the inner product  
\[\left<\,f \,\otimes\, g \,,\, f^{\,\prime} \,\otimes\, g^{\,\prime}\,\right> \,=\, \left<\,f  \,,\, f^{\,\prime}\,\right>_{1}\;\left<\,g  \,,\, g^{\,\prime}\,\right>_{2}\; \;\forall\; f,\, f^{\,\prime} \,\in\, H\; \;\&\; \;g,\, g^{\,\prime} \,\in\, K.\]
The norm on \,$H \,\otimes\, K$\, is defined by 
\begin{equation}\label{eqn1.01}
\left\|\,f \,\otimes\, g\,\right\| \,=\, \|\,f\,\|_{1}\;\|\,g\,\|_{2}\; \;\forall\; f \,\in\, H\; \;\&\; \,g \,\in\, K.
\end{equation}
The space \,$H \,\otimes\, K$\, is complete with respect to the above inner product.\;Therefore, the space \,$H \,\otimes\, K$\, is an Hilbert space.\\     
\end{definition} 

For \,$Q \,\in\, \mathcal{B}\,(\,H\,)$\, and \,$T \,\in\, \mathcal{B}\,(\,K\,)$, the tensor product of operators \,$Q$\, and \,$T$\, is denoted by \,$Q \,\otimes\, T$\, and defined as 
\[\left(\,Q \,\otimes\, T\,\right)\,A \,=\, Q\,A\,T^{\,\ast}\; \;\forall\; \;A \,\in\, H \,\otimes\, K.\]

\begin{theorem}\cite{Folland, Li}\label{th1.1}
Let \,$Q,\, Q^{\prime} \,\in\, \mathcal{B}\,(\,H\,)$\, and \,$T,\, T^{\prime} \,\in\, \mathcal{B}\,(\,K\,)$.\,Then \begin{itemize}
\item[(I)]\hspace{.2cm} \,$Q \,\otimes\, T \,\in\, \mathcal{B}\,(\,H \,\otimes\, K\,)$\, and \,$\left\|\,Q \,\otimes\, T\,\right\| \,=\, \|\,Q\,\|\; \|\,T\,\|$.
\item[(II)]\hspace{.2cm} \,$\left(\,Q \,\otimes\, T\,\right)\,(\,f \,\otimes\, g\,) \,=\, Q\,(\,f\,) \,\otimes\, T\,(\,g\,)$\, for all \,$f \,\in\, H,\, g \,\in\, K$.
\item[(III)]\hspace{.2cm} $\left(\,Q \,\otimes\, T\,\right)\,\left(\,Q^{\,\prime} \,\otimes\, T^{\,\prime}\,\right) \,=\, (\,Q\,Q^{\,\prime}\,) \,\otimes\, (\,T\,T^{\,\prime}\,)$. 
\item[(IV)]\hspace{.2cm} \,$Q \,\otimes\, T$\, is invertible if and only if \,$Q$\, and \,$T$\, are invertible, in which case \,$\left(\,Q \,\otimes\, T\,\right)^{\,-\, 1} \,=\, \left(\,Q^{\,-\, 1} \,\otimes\, T^{\,-\, 1}\,\right)$.
\item[(V)]\hspace{.2cm} \,$\left(\,Q \,\otimes\, T\,\right)^{\,\ast} \,=\, \left(\,Q^{\,\ast} \,\otimes\, T^{\,\ast}\,\right)$.  
\item[(VI)]\hspace{.2cm} Let \,$f,\, f^{\,\prime} \,\in\, H \,\setminus\, \{\,0\}$\, and \,$g,\, g^{\,\prime} \,\in\, K \,\setminus\, \{\,0\}$\,.\;If \,$f \,\otimes\, g \,=\, f^{\,\prime} \,\otimes\, g^{\,\prime}$, then there exist constants \,$a$\, and \,$b$\, with \,$a\,b \,=\, 1$\, such that \,$f \,=\, a\,f^{\,\prime}$\, and \,$g \,=\, b\,g^{\,\prime}$.  
\end{itemize}
\end{theorem}

\section{Fusion frame in tensor product of Hilbert spaces}

\smallskip\hspace{.6 cm} In this section, fusion frame in the Hilbert space \,$H \,\otimes\, K$\, have been discussed and some results associated to fusion frame in \,$H \,\otimes\, K$\, are likely to be established.\;At the end, we discuss the relationship among the resolution of the identity operator on \,$H \,\otimes\, K$\, and resolutions of the identity operators on \,$H$\, and \,$K$, respectively.   

\begin{definition}
Let \,$\left\{\,v_{\,i}\,\right\}_{\, i \,\in\, I}$\, and \,$\left\{\,w_{\,j}\,\right\}_{ j \,\in\, J}$\, be two families of positive weights i\,.\,e., \,$v_{\,i} \,>\, 0\, \;\forall\; i \,\in\, I$\, and \,$w_{\,j} \,>\, 0\, \;\forall\; j \,\in\, J$\, and \,$\left\{\,V_{i} \,\otimes\, W_{j} \,:\, (\,i \,,\, j\,) \,\in\, I \,\times\, J\,\right\}$\, be a family of closed subspaces of \,$H \,\otimes\, K$.\,Then the family \,$V \otimes\, W =\left\{\,\left(\,V_{i} \otimes W_{j},\, v_{\,i}\,w_{\,j}\,\right)\,\right\}_{i,\, j}$\, is called a fusion frame for \,$H \otimes K$, if there exist \,$0 < A \leq B < \infty$\, such that
\[A\, \left\|\,f \otimes g\,\right\|^{\,2} \,\leq\, \sum\limits_{i,\, j}\,v^{\,2}_{\,i}\,w^{\,2}_{j}\,\left\|\,P_{\,V_{i} \,\otimes\, W_{j}}\,(\,f \otimes g\,)\,\right\|^{\,2} \,\leq\, B\, \left\|\,f \otimes g\,\right\|^{\,2}\; \;\forall\; f \otimes g \,\in\, H \,\otimes\, K,\]
where \,$P_{\,V_{i} \,\otimes\, W_{j}}$\, is the orthogonal projection of \,$H \,\otimes\, K$\, onto \,$V_{i} \,\otimes\, W_{j}$.\;The constants \,$A$\, and \,$B$\, are called the frame bounds.\;If \,$A \,=\, B$\, then it is called a tight fusion frame for \,$H \,\otimes\, K$.\;If the family \,$V \,\otimes\, W$\, satisfies the inequality
\[\sum\limits_{i,\, j}\,v^{\,2}_{\,i}\,w^{\,2}_{j}\,\left\|\,P_{\,V_{i} \,\otimes\, W_{j}}\,(\,f \,\otimes\, g\,)\,\right\|^{\,2} \,\leq\, B\, \left\|\,f \,\otimes\, g\,\right\|^{\,2}\; \;\forall\; f \,\otimes\, g \,\in\, H \,\otimes\, K,\]
then it is called a fusion Bessel sequence in \,$H \,\otimes\, K$\, with bound \,$B$.    
\end{definition}

\begin{definition}
For \,$i \,\in\, I$\, and \,$j \,\in\, J$, define the space \,$l^{\,2}\,\left(\,\left\{\,V_{i} \,\otimes\, W_{j}\,\right\}\,\right)$
\[\,=\, \left\{\,\left\{\,f_{\,i} \,\otimes\, g_{\,j}\,\right\} \,:\, f_{\,i} \,\otimes\, g_{\,j} \,\in\, V_{i} \,\otimes\, W_{j}, \;\text{and}\; \;\sum\limits_{i,\, j}\,\left\|\,f_{\,i} \,\otimes\, g_{\,j}\,\right\|^{\,2} \,<\, \infty\,\right\}\]
with inner product 
\[\left<\,\left\{\,f_{\,i} \,\otimes\, g_{\,j}\,\right\} \,,\, \left\{\,f^{\,\prime}_{\,i} \,\otimes\, g^{\,\prime}_{\,j}\,\right\}\,\right>_{l^{\,2}} \,=\, \sum\limits_{i,\, j}\,\left<\,f_{\,i} \,\otimes\, g_{\,j}\, \,,\, f^{\,\prime}_{\,i} \,\otimes\, g^{\,\prime}_{\,j}\,\right>\hspace{2.8cm}\]
\[=\,\sum\limits_{i,\, j}\,\left<\,f_{\,i} \,,\, f^{\,\prime}_{\,i}\,\right>_{1}\,\left<\,g_{\,j} \,,\, g^{\,\prime}_{\,j}\,\right>_{2} \,=\, \left(\,\sum\limits_{\,i \,\in\, I}\,\left<\,f_{\,i} \,,\, f^{\,\prime}_{\,i}\,\right>_{1}\,\right)\,\left(\,\sum\limits_{\,j \,\in\, J}\,\left<\,g_{\,j} \,,\, g^{\,\prime}_{\,j}\,\right>_{2}\,\right)\hspace{.7cm}\]
\[=\, \left<\,\left\{\,f_{\,i}\,\right\}_{ i \,\in\, I} \,,\, \left\{\,f^{\,\prime}_{\,i}\,\right\}_{ i \,\in\, I}\,\right>_{l^{\,2}\left(\,\left\{\,V_{i}\,\right\}_{ i \,\in\, I}\,\right)}\,\left<\,\left\{\,g_{\,j}\,\right\}_{ j \,\in\, J} \,,\, \left\{\,g^{\,\prime}_{\,j}\,\right\}_{ j \,\in\, J}\,\right>_{l^{\,2}\left(\,\left\{\,W_{j}\,\right\}_{j \,\in\, J}\,\right)}.\]
It is easy to verify that the space \,$l^{\,2}\,\left(\,\left\{\,V_{i} \,\otimes\, W_{j}\,\right\}\,\right)$\, is an Hilbert space with respect to the above inner product.
\end{definition}

\begin{note}\label{note1.1}
Since \,$\left\{\,V_{i}\,\right\}_{ i \,\in\, I},\; \left\{\,W_{j}\,\right\}_{ j \,\in\, J}$\, and \,$\left\{\,V_{i} \,\otimes\, W_{j} \,\right\}_{i,\, j}$\, are the families of closed subspaces of \,$H,\, K$\, and \,$H \,\otimes\, K$\, respectively, it is easy to verify that \,$P_{\,V_{i} \,\otimes\, W_{j}} \,=\, P_{\,V_{i}} \,\otimes\, P_{\,W_{j}}$.
\end{note}

For the remaining part of this paper, we denote the collections \,$\left\{\,\left(\,V_{i},\, v_{\,i}\,\right)\,\right\}_{\, i \,\in\, I}$, \,$\left\{\,\left(\,W_{j},\, w_{\,j}\,\right)\,\right\}_{ j \,\in\, J},\,\left\{\,\left(\,V_{i} \,\otimes\, W_{j},\, v_{\,i}\,w_{\,j}\,\right)\,\right\}_{i,\, j}$\, and \,$\left\{\,\left(\,V^{\,\prime}_{i} \,\otimes\, W^{\,\prime}_{j},\, v^{\,\prime}_{\,i}\,w^{\,\prime}_{\,j}\,\right)\,\right\}_{i,\, j}$\, by \,$V,\, W$, \,$V \,\otimes\, W$\, and \,$V^{\,\prime} \,\otimes\, W^{\,\prime}$, respectively.\\   

\begin{theorem}\label{thm2}
Let \,$V$\, and \,$W$\, be the families of weighted closed subspaces in \,$H$\, and \,$K$, respectively.\;Then \,$V$\, and \,$W$\, are fusion frames for \,$H$\, and \,$K$\,  if and only if \,$V \,\otimes\, W$\, is a fusion frame for \,$H \,\otimes\, K$. 
\end{theorem}

\begin{proof}
First we suppose that \,$V$\, and \,$W$\, are fusion frames for \,$H$\, and \,$K$.\;Then there exist positive constants \,$(\,A,\, B\,)$\, and \,$(\,C,\, D\,)$\, such that
\begin{equation}\label{eq1}
A\,\left\|\,f \,\right\|_{\,1}^{\,2} \,\leq\, \sum\limits_{\,i \,\in\, I}\, v_{i}^{\,2} \,\left\|\, P_{\,V_{i}}\,(\,f\,) \,\right\|_{1}^{\,2}  \,\leq\, B\,\left\|\, f \, \right\|_{1}^{\,2}\; \;\forall\; f \,\in\, H
\end{equation}
\begin{equation}\label{eq1.1}
C\,\left\|\,g \,\right\|_{2}^{\,2} \,\leq\, \sum\limits_{\,j \,\in\, J}\, w_{j}^{\,2}\, \left\|\, P_{\,W_{j}}\,(\,g\,) \,\right\|_{2}^{\,2} \,\leq\, D\,\left\|\, g \, \right\|_{2}^{\,2}\; \;\forall\; g \,\in\, K.
\end{equation}
Multiplying (\ref{eq1}) and (\ref{eq1.1}), and using the definition of norm on \,$H \,\otimes\, K$, we get
\[A\,C\left\|\,f\,\right\|_{1}^{\,2}\,\left\|\,g\,\right\|_{2}^{\,2} \,\leq \left(\sum\limits_{\,i \,\in\, I} v_{i}^{\,2}\left\|\, P_{\,V_{i}}\,(\,f\,)\,\right\|_{1}^{\,2}\right)\left(\,\sum\limits_{\,j \,\in\, J} w_{j}^{\,2} \left\|\, P_{\,W_{j}}\,(\,g\,) \,\right\|_{2}^{\,2}\right) \,\leq B\,D\left\|\,f \,\right\|_{1}^{\,2}\,\left\|\,g \,\right\|_{2}^{\,2}\]
\[\Rightarrow\, A\,C\,\left\|\,f \,\otimes\, g\,\right\|^{\,2} \,\leq\, \sum\limits_{i,\, j}\,v^{\,2}_{\,i}\,w^{\,2}_{\,j}\,\left\|\, P_{\,V_{i}}\,(\,f\,) \,\right\|_{1}^{\,2}\,\left\|\, P_{\,W_{j}}\,(\,g\,) \,\right\|_{2}^{\,2} \,\leq\, B\,D\,\left\|\,f \,\otimes\, g\,\right\|^{\,2}\]
\[\Rightarrow\, A\,C\,\left\|\,f \,\otimes\, g\,\right\|^{\,2} \,\leq\, \sum\limits_{i,\, j}\,v^{\,2}_{\,i}\,w^{\,2}_{\,j}\,\left\|\, P_{\,V_{i}}\,(\,f\,) \otimes\, P_{\,W_{j}}\,(\,g\,) \,\right\|^{\,2} \,\leq\, B\,D\,\left\|\,f \,\otimes\, g\,\right\|^{\,2}.\hspace{.1cm}\]
Therefore, for all \,$f \,\otimes\, g \,\in\, H \,\otimes\, K$, we have
\[A\,C\,\left\|\,f \,\otimes\, g\,\right\|^{\,2} \,\leq\, \sum\limits_{i,\, j}\,v^{\,2}_{\,i}\,w^{\,2}_{\,j}\,\left\|\, \left(\,P_{\,V_{i}} \,\otimes\, P_{\,W_{j}}\,\right)\,(\,f \,\otimes\, g\,) \,\right\|^{\,2} \,\leq\, B\,D\,\left\|\,f \,\otimes\, g\,\right\|^{\,2}\]
\[\Rightarrow\, A\,C\,\left\|\,f \,\otimes\, g\,\right\|^{\,2} \,\leq\, \sum\limits_{i,\, j}\,v^{\,2}_{\,i}\,w^{\,2}_{\,j}\,\left\|\,P_{\,V_{i} \,\otimes\, W_{j}}\,(\,f \,\otimes\, g\,)\,\right\|^{\,2} \,\leq\, B\,D\,\left\|\,f \,\otimes\, g\,\right\|^{\,2}.\hspace{.53cm}\]  
This shows that \,$V \,\otimes\, W$\, is a fusion frame for \,$H \,\otimes\, K$\, with bounds \,$A\,C$\, and \,$B\,D$. \\

Conversely, suppose that \,$V \,\otimes\, W$\, is a fusion frame for \,$H \,\otimes\, K$\, with bounds \,$A\, \;\text{and}\; \,B$.\;Then, for each \,$f \,\otimes\, g \,\in\, H \,\otimes\, K \,-\, \{\,\theta \,\otimes\, \theta\,\}$, we have
\[A\, \left\|\,f \,\otimes\, g\,\right\|^{\,2} \,\leq\, \sum\limits_{i,\, j}\,v^{\,2}_{\,i}\,w^{\,2}_{j}\,\left\|\,P_{\,V_{i} \,\otimes\, W_{j}}\,(\,f \,\otimes\, g\,)\,\right\|^{\,2} \,\leq\, B\, \left\|\,f \,\otimes\, g\,\right\|^{\,2}\]
\[\Rightarrow\, A\,\left\|\,f \,\right\|_{\,1}^{\,2}\,\left\|\,g \,\right\|_{\,2}^{\,2} \,\leq\, \sum\limits_{i,\, j}\,v^{\,2}_{\,i}\,w^{\,2}_{\,j}\,\left\|\, P_{\,V_{i}}\,(\,f\,) \otimes\, P_{\,W_{j}}\,(\,g\,) \,\right\|^{\,2} \,\leq\, B\,\left\|\,f \,\right\|_{\,1}^{\,2}\,\left\|\,g \,\right\|_{\,2}^{\,2}\hspace{1.6cm}\]
\[\Rightarrow A\left\|\,f\,\right\|_{1}^{\,2}\,\left\|\,g\,\right\|_{2}^{\,2} \,\leq \left(\sum\limits_{\,i \,\in\, I} v_{i}^{\,2}\left\|\, P_{\,V_{i}}\,(\,f\,)\,\right\|_{1}^{\,2}\right)\left(\,\sum\limits_{\,j \,\in\, J} w_{j}^{\,2} \left\|\, P_{\,W_{j}}\,(\,g\,) \,\right\|_{2}^{\,2}\right) \,\leq B\left\|\,f \,\right\|_{1}^{\,2}\,\left\|\,g \,\right\|_{2}^{\,2}.\]
Since \,$f \,\otimes\, g$\, is non-zero vector, \,$f$\, and \,$g$\, are also non-zero vectors and therefore \,$\sum\limits_{\,i \,\in\, I}\, v_{i}^{\,2} \,\left\|\, P_{\,V_{i}}\,(\,f\,) \,\right\|_{1}^{\,2}$\, and \,$\sum\limits_{\,j \,\in\, J}\, w_{j}^{\,2}\, \left\|\, P_{\,W_{j}}\,(\,g\,) \,\right\|_{2}^{\,2}$\, are non-zero.
\[\Rightarrow\, \dfrac{A\,\|\,g\,\|_{2}^{\,2}}{\sum\limits_{\,j \,\in\, J}\, w_{j}^{\,2}\, \left\|\, P_{\,W_{j}}\,(\,g\,) \,\right\|_{2}^{\,2}}\,\left\|\,f \,\right\|_{1}^{\,2}\, \leq\, \sum\limits_{\,i \,\in\, I}\, v_{i}^{\,2} \,\left\|\, P_{\,V_{i}}\,(\,f\,) \,\right\|_{1}^{\,2} \,\leq\, \dfrac{B\,\|\,g\,\|_{2}^{\,2}}{\sum\limits_{\,j \,\in\, J}\, w_{j}^{\,2}\, \left\|\, P_{\,W_{j}}\,(\,g\,) \,\right\|_{2}^{\,2}}\,\left\|\,f \,\right\|_{1}^{\,2}\]
\[\Rightarrow\, A_{\,1} \,\left\|\,f \,\right\|_{1}^{\,2} \,\leq\, \sum\limits_{\,i \,\in\, I}\, v_{i}^{\,2} \,\left\|\,P_{\,V_{i}}\,(\,f\,) \,\right\|_{1}^{\,2}  \,\leq\, B_{\,1}\, \left\|\, f \, \right\|_{1}^{\,2}\; \;\forall\; f \,\in\, H,\]
where \,$A_{\,1} \,=\, \dfrac{A\,\|\,g\,\|_{2}^{\,2}}{\sum\limits_{\,j \,\in\, J}\, w_{j}^{\,2}\, \left\|\, P_{\,W_{j}}\,(\,g\,) \,\right\|_{2}^{\,2}}$\; and \,\,$B_{\,1} \,=\, \dfrac{B\,\|\,g\,\|_{2}^{\,2}}{\sum\limits_{\,j \,\in\, J}\, w_{\,j}^{\,2}\, \left\|\, P_{\,W_{j}}\,(\,g\,)\,\right\|_{2}^{\,2}}$.\;This shows that  \,$V$\, is a fusion frame for \,$H$.\;Similarly, it can be shown that \,$W$\, is a fusion frame for \,$K$.          
\end{proof}

\begin{note}
Let \,$V \,\otimes\, W$\, be a fusion frame for \,$H \,\otimes\, K$.\;According to the definition (\ref{def1}), the corresponding frame operator \,$S_{\,V \,\otimes\, W} \,:\, H \,\otimes\, K \,\to\, H \,\otimes\, K$\, is given by
\[S_{\,V \,\otimes\, W}\,(\,f \,\otimes\, g\,) \,=\, \sum\limits_{i,\, j}\,v^{\,2}_{\,i}\,w^{\,2}_{\,j}\,P_{\,V_{i} \,\otimes\, W_{j}}\,(\,f \,\otimes\, g\,)\; \;\forall\; f \,\otimes\, g \,\in\, H \,\otimes\, K.\]  
\end{note}

\begin{theorem}
Let \,$S_{V},\, S_{W}$\, and \,$S_{\,V \,\otimes\, W}$\, be the corresponding frame operators for the fusion frames \,$V$, \,$W$\, and \,$V \,\otimes\, W$, respectively.\,Then \,$S_{\,V \,\otimes\, W} \,=\, S_{\,V} \,\otimes\, S_{\,W}\; \;\text{and}\;\; \;S^{\,-\, 1}_{\,V \,\otimes\, W} \,=\, S^{\,-\, 1}_{\,V} \,\otimes\, S^{\,-\, 1}_{\,W}$.
\end{theorem}

\begin{proof}
For each \,$f \,\otimes\, g \,\in\, H \,\otimes\, K$, we have
\[S_{\,V \,\otimes\, W}\,(\,f \,\otimes\, g\,) \,=\, \sum\limits_{i,\, j}\,v^{\,2}_{\,i}\,w^{\,2}_{\,j}\,P_{\,V_{i} \,\otimes\, W_{j}}\,(\,f \,\otimes\, g\,)\]
\[ \,=\, \sum\limits_{i,\, j}\,v^{\,2}_{\,i}\,w^{\,2}_{\,j}\,\left(\,P_{\,V_{i}} \,\otimes\, P_{\,W_{j}}\,\right)\,(\,f \,\otimes\, g\,) \,=\, \sum\limits_{i,\, j}\,v^{\,2}_{\,i}\,w^{\,2}_{\,j}\,\left(\,P_{\,V_{i}}\,(\,f\,) \,\otimes\, P_{\,W_{j}}\,(\,g\,)\,\right)\hspace{3cm}\]
\[ \,=\, \left(\sum\limits_{\,i \,\in\, I} v_{i}^{\,2}P_{\,V_{i}}\,(\,f\,)\right) \otimes \left(\sum\limits_{\,j \,\in\, J} w_{j}^{\,2}P_{\,W_{j}}\,(\,g\,)\right) = S_{V}\,(\,f\,) \,\otimes\, S_{W}\,(\,g\,) = S_{V} \otimes S_{W}\,(\,f \,\otimes\, g\,).\]
This implies that \,$S_{\,V \,\otimes\, W} \,=\, S_{V} \,\otimes\, S_{W}$.\;Since \,$S_{\,V}\, \;\text{and}\; \,S_{\,W}$\, are invertible, by \,$(\,IV\,)$\, of the Theorem (\ref{th1.1}), it follows that \,$S^{\,-\, 1}_{\,V \,\otimes\, W} \,=\, S^{\,-\, 1}_{V} \,\otimes\, S^{\,-\, 1}_{W}$.     
\end{proof}

\begin{theorem}\label{th2}
Let \,$V$\, and \,$W$\, be fusion frames for \,$H$\, and \,$K$\, with frame bounds \,$(\,A,\, B\,)$\, and \,$(\,C,\, D\,)$\, having their associated frame operators \,$S_{V}$\, and \,$S_{W}$, respectively.\;If \,$T_{\,1}$\, and \,$T_{\,2}$\, are invertible and unitary operators on \,$H$\, and \,$K$, then \,$\Delta \,=\, \left\{\,\left(\,T_{\,1} \,\otimes\, T_{\,2}\,\right)\,\left(\,V_{i} \,\otimes\, W_{j}\,\right),\, v_{\,i}\,w_{\,j}\,\right\}_{i,\, j}$\, is a fusion frame for \,$H \,\otimes\, K$.   
\end{theorem}

\begin{proof}
Since \,$T_{\,1}\, \;\text{and}\; \,T_{\,2}$\, are invertible, by \,$(IV)$\, of the Theorem (\ref{th1.1}), \,$T_{\,1} \,\otimes\, T_{\,2}$\, is invertible and \,$\left(\,T_{\,1} \,\otimes\, T_{\,2}\,\right)^{\,-\, 1} \,=\, \left(\,T_{\,1}^{\,-\, 1} \,\otimes\, T_{\,2}^{\,-\, 1}\,\right)$.\;Also, by Theorem (\ref{th0.001}), for any \,$i \,\in\, I\; \;\text{and}\, \,j \,\in\, J$, we get
\begin{equation}\label{eqn2}
\left\|\,P_{\,V_{i}}\,T^{\,\ast}_{\,1}\,(\,f\,)\,\right\|_{1} \,\leq\, \left\|\,T^{\,\ast}_{\,1}\,\right\|\; \left\|\,P_{\,T_{\,1}\,V_{i}}\,(\,f\,)\,\right\|_{1}\; \;\forall\; f \,\in\, H, \;\text{and}
\end{equation}
\begin{equation}\label{eq2}
 \;\left\|\,P_{\,W_{j}}\,T^{\,\ast}_{\,2}\,(\,g\,)\,\right\|_{2} \,\leq\, \left\|\,T^{\,\ast}_{\,2}\,\right\|\; \left\|\,P_{\,T_{\,2}\,W_{j}}\,(\,g\,)\,\right\|_{2}\; \;\forall\; g \,\in\, K.
\end{equation}
Again, since \,$T_{\,1}\, \;\text{and}\; \,T_{\,2}$\, are invertible, for each \,$f \,\in\, H$\, and \,$g \,\in\, K$, we obtain 
\begin{equation}\label{eq2.1}
\|\,f\,\|_{1} \,\leq\, \left\|\,\left(\,T_{\,1}^{\,-\, 1}\,\right)^{\,\ast}\,\right\|\; \left\|\,T^{\,\ast}_{\,1}\,(\,f\,)\,\right\|_{1}\; \;\&\; \; \|\,g\,\|_{2} \,\leq\, \left\|\,\left(\,T_{\,2}^{\,-\, 1}\,\right)^{\,\ast}\,\right\|\; \left\|\,T^{\,\ast}_{\,2}\,(\,g\,)\,\right\|_{2}.
\end{equation}
Now, for each \,$f \,\otimes\, g \,\in\, H \,\otimes\, K$, using Theorem (\ref{th1.1}), we get
\[\sum\limits_{i,\, j}v^{\,2}_{\,i}\,w^{\,2}_{\,j}\left\|\,P_{\,\left(\,T_{\,1} \,\otimes\, T_{\,2}\,\right)\,\left(\,V_{i} \,\otimes\, W_{j}\,\right)}\,(\,f \,\otimes\, g\,)\,\right\|^{\,2} = \sum\limits_{i,\, j}v^{\,2}_{\,i}\,w^{\,2}_{\,j}\left\|\,P_{\,\left(\,T_{\,1}\,V_{i} \,\otimes\, T_{\,2}\,W_{j}\,\right)}\,(\,f \,\otimes\, g\,)\,\right\|^{\,2}\]
\[=\, \sum\limits_{i,\, j}\,v^{\,2}_{\,i}\,w^{\,2}_{\,j}\,\left\|\,\left(\,P_{\,T_{\,1}\,V_{i}} \,\otimes\, P_{\,T_{\,2}\,W_{j}}\,\right)\,(\,f \,\otimes\, g\,)\,\right\|^{\,2}\; \;[\;\text{by note (\ref{note1.1})}\;]\hspace{3cm}\]
\begin{equation}\label{eq2.2}
=\, \left(\,\sum\limits_{\,i \,\in\, I}\, v^{\,2}_{\,i}\,\left\|\,P_{\,T_{\,1}\,V_{i}}\,(\,f\,)\,\right\|_{1}^{\,2}\,\right)\,\left(\,\sum\limits_{\,j \,\in\, J}\, w^{\,2}_{\,j}\,\left\|\,P_{\,T_{\,2}\,W_{j}}\,(\,g\,)\,\right\|_{2}^{\,2}\,\right)\; \;[\;\text{using (\ref{eqn1.01})}\;]\hspace{.5cm}
\end{equation}
\[\geq\, \dfrac{1}{\left\|\,T_{\,1}\,\right\|^{\,2}\,\left\|\,T_{\,2}\,\right\|^{\,2}} \left(\,\sum\limits_{\,i \,\in\, I}\, v^{\,2}_{\,i}\,\left\|\,P_{\,V_{i}}\left(\,T_{\,1}^{\,\ast}\,f\,\right)\,\right\|_{1}^{\,2}\,\right) \left(\,\sum\limits_{\,j \,\in\, J}\, w^{\,2}_{\,j}\,\left\|\,P_{\,W_{j}}\left(\,T^{\,\ast}_{\,2}\,g\,\right)\,\right\|_{2}^{\,2}\,\right)\; \;[\;\text{by (\ref{eqn2}) \& (\ref{eq2})}\;]\]
\[\geq\, \dfrac{A\,C}{\left\|\,T_{\,1}\,\right\|^{\,2}\,\left\|\,T_{\,2}\,\right\|^{\,2}}\,\left\|\,T^{\,\ast}_{\,1}\,(\,f\,)\,\right\|_{1}^{\,2}\;\left\|\,T^{\,\ast}_{\,2}\,(\,g\,)\,\right\|_{2}^{\,2}\; \;[\;\text{since $V, \;W$ are fusion frames}\;]\hspace{1.5cm}\]
\[ \,\geq\, \dfrac{A\,C}{\left\|\,T_{\,1}\,\right\|^{\,2}\,\left\|\,T_{\,2}\,\right\|^{\,2}\,\left\|\,T^{\,-\, 1}_{\,1}\,\right\|^{\,2}\,\left\|\,T^{\,-\, 1}_{\,2}\,\right\|^{\,2}}\,\|\,f\,\|^{\,2}_{1}\,\|\,g\,\|^{\,2}_{2}\; \;[\;\text{by (\ref{eq2.1})}\;]\hspace{5.9cm}\]
\[=\, \dfrac{A\,C}{\left\|\,T_{\,1} \,\otimes\, T_{\,2}\,\right\|^{\,2}\,\left\|\,\left(\,T_{\,1} \,\otimes\, T_{\,2}\,\right)^{\,-\, 1}\,\right\|^{\,2}}\;\left\|\,f \,\otimes\, g\,\right\|^{\,2}.\hspace{5.7cm}\]
On the other hand, since \,$T_{\,1}\, \;\text{and}\; \,T_{\,2}$\, are unitary operators, again by Theorem (\ref{th0.001}), \,$P_{\,T_{\,1}\,V_{i}}\,T_{\,1} \,=\, T_{\,1}\,P_{\,V_{i}}\, \;\text{and}\; \,P_{\,T_{\,2}\,W_{j}}\,T_{\,2} \,=\, T_{\,2}\,P_{\,W_{j}}$.\;Then, for all \,$f \,\otimes\, g \,\in\, H \,\otimes\, K$,  
\[\sum\limits_{i,\, j}\,v^{\,2}_{\,i}\,w^{\,2}_{\,j}\,\left\|\,P_{\,\left(\,T_{\,1} \,\otimes\, T_{\,2}\,\right)\,\left(\,V_{i} \,\otimes\, W_{j}\,\right)}\,(\,f \,\otimes\, g\,)\,\right\|^{\,2}\]
\[=\, \left(\,\sum\limits_{\,i \,\in\, I}\, v^{\,2}_{\,i}\,\left\|\,P_{\,T_{\,1}\,V_{i}}\,(\,f\,)\,\right\|_{1}^{\,2}\,\right)\,\left(\,\sum\limits_{\,j \,\in\, J}\, w^{\,2}_{\,j}\,\left\|\,P_{\,T_{\,2}\,W_{j}}\,(\,g\,)\,\right\|_{2}^{\,2}\,\right)\; \;[\;\text{by (\ref{eq2.2})}\;]\hspace{2.6cm}\]
\[=\, \left(\,\sum\limits_{\,i \,\in\, I}\, v^{\,2}_{\,i}\,\left\|\,T_{\,1}\;P_{\,V_{i}}\left(\,T^{\,-\, 1}_{\,1}\,f\,\right)\,\right\|_{1}^{\,2}\,\right)\,\left(\,\sum\limits_{\,j \,\in\, J}\, w^{\,2}_{\,j}\,\left\|\,T_{\,2}\;P_{\,W_{j}}\left(\,T^{\,-\, 1}_{\,2}\,g\,\right)\,\right\|_{2}^{\,2}\,\right)\hspace{1.9cm}\]
\[\,\leq\, \left\|\,T_{\,1}\,\right\|^{\,2}\,\left\|\,T_{\,2}\,\right\|^{\,2}\,\left(\,\sum\limits_{\,i \,\in\, I}\, v^{\,2}_{\,i}\,\left\|\,P_{\,V_{i}}\left(\,T^{\,-\, 1}_{\,1}\,f\,\right)\,\right\|_{1}^{\,2}\,\right)\,\left(\,\sum\limits_{\,j \,\in\, J}\, w^{\,2}_{\,j}\,\left\|\,P_{\,W_{j}}\left(\,T^{\,-\, 1}_{\,2}\,g\,\right)\,\right\|_{2}^{\,2}\,\right)\hspace{2cm}\]
\[\leq\, B\,D\,\left\|\,T_{\,1}\,\right\|^{\,2}\,\left\|\,T_{\,2}\,\right\|^{\,2}\,\left\|\,T^{\,-\, 1}_{1}\,(\,f\,)\,\right\|^{\,2}_{\,1}\,\left\|\,T^{\,-\, 1}_{\,2}\,(\,g\,)\,\right\|^{\,2}_{2}\; \;[\;\text{since $V,\;W$ are fusion frames}\;]\]
\[ \,\leq\, B\,D\,\left\|\,T_{\,1}\,\right\|^{\,2}\,\left\|\,T_{\,2}\,\right\|^{\,2}\,\left\|\,T^{\,-\, 1}_{\,1}\,\right\|^{\,2}\,\left\|\,T^{\,-\, 1}_{\,2}\,\right\|^{\,2}\,\|\,f\,\|^{\,2}_{\,1}\,\|\,g\,\|^{\,2}_{\,2}\hspace{5cm}\]
\[=\, B\,D\,\left\|\,T_{\,1} \,\otimes\, T_{\,2}\,\right\|^{\,2}\,\left\|\,\left(\,T_{\,1} \,\otimes\, T_{\,2}\,\right)^{\,-\, 1}\,\right\|^{\,2}\,\left\|\,f \,\otimes\, g\,\right\|^{\,2}.\hspace{5cm}\]
Hence, \,$\Delta$\, is a fusion frame for \,$H \,\otimes\, K$.   
\end{proof}

\begin{theorem}
The corresponding fusion frame operator for the fusion frame \,$\Delta$\, is \,$\left(\,T_{\,1} \,\otimes\, T_{\,2}\,\right)\,S_{V \,\otimes\, W}\,\left(\,T_{\,1} \,\otimes\, T_{\,2}\,\right)^{\,-\, 1}$.
\end{theorem}

\begin{proof}
For each \,$f \,\otimes\, g \,\in\, H \,\otimes\, K$, we have
\[\sum\limits_{i,\, j}\,v^{\,2}_{\,i}\,w^{\,2}_{\,j}\,P_{\,\left(\,T_{\,1} \,\otimes\, T_{\,2}\,\right)\,\left(\,V_{i} \,\otimes\, W_{j}\,\right)}\,(\,f \,\otimes\, g\,) \,=\, \sum\limits_{i,\, j}\,v^{\,2}_{\,i}\,w^{\,2}_{\,j}\,\left(\,P_{\,T_{\,1}\,V_{i}} \,\otimes\, P_{\,T_{\,2}\,W_{j}}\,\right)\,(\,f \,\otimes\, g\,)\]
\[ =\, \left(\,\sum\limits_{\,i \,\in\, I}\, v^{\,2}_{\,i}\,P_{\,T_{\,1}\,V_{i}}\,(\,f\,)\,\right) \,\otimes\, \left(\,\sum\limits_{\,j \,\in\, J}\, w^{\,2}_{\,j}\,P_{\,T_{\,2}\,W_{j}}\,(\,g\,)\,\right)\hspace{5cm}\]
\[\,=\, \left(\,\sum\limits_{\,i \,\in\, I}\, v^{\,2}_{\,i}\;T_{\,1}\,P_{\,V_{i}}\left(\,T^{\,-\, 1}_{\,1}\,f\,\right)\,\right) \,\otimes\, \left(\,\sum\limits_{\,j \,\in\, J}\, w^{\,2}_{\,j}\;T_{\,2}\,P_{\,W_{j}}\left(\,T^{\,-\, 1}_{\,2}\,g\,\right)\,\right)\;[\;\text{by Theorem (\ref{th0.001})}\;]\]
\[ \,=\, T_{\,1}\,S_{V}\,\left(\,T^{\,-\, 1}_{\,1}\,(\,f\,)\,\right) \,\otimes\, T_{\,2}\,S_{W}\,\left(\,T^{\,-\, 1}_{\,2}\,(\,g\,)\,\right)\hspace{7cm}\]
\[ \,=\, \left(\,T_{\,1} \,\otimes\, T_{\,2}\,\right)\,\left(\,S_{V} \,\otimes\, S_{W}\,\right)\,\left(\,T_{\,1}^{\,-\, 1} \,\otimes\, T_{\,2}^{\,-\, 1}\,\right)\,(\,f \,\otimes\, g\,)\;[\;\text{by Theorem (\ref{th1.1})}\;]\hspace{1.5cm}\]
\[\,=\, \left(\,T_{\,1} \,\otimes\, T_{\,2}\,\right)\,S_{V \,\otimes\, W}\,\left(\,T_{\,1} \,\otimes\, T_{\,2}\,\right)^{\,-\, 1}\,(\,f \,\otimes\, g\,).\hspace{6.5cm}\]  
This shows that \,$\left(\,T_{\,1} \,\otimes\, T_{\,2}\,\right)\,S_{V \,\otimes\, W}\,\left(\,T_{\,1} \,\otimes\, T_{\,2}\,\right)^{\,-\, 1}$\, is the corresponding fusion frame operator for \,$\Delta$.
\end{proof}

\begin{definition}
A family of bounded operators \,$\left\{\,T_{i} \,\otimes\, U_{j}\,\right\}_{i,\, j}$\, on  a tensor product of Hilbert space \,$H \,\otimes\, K$\, is called a resolution of the identity operator on \,$H \,\otimes\, K$, if for all \,$f \,\otimes\, g \,\in\, H \,\otimes\, K$, we have 
\[f \,\otimes\, g \,=\, \sum\limits_{i,\, j}\,\left(\,T_{i} \,\otimes\, U_{j}\,\right)\,(\,f \,\otimes\, g\,),\]
provided the series converges unconditionally for all \,$f \,\otimes\, g \,\in\, H \,\otimes\, K$. 
\end{definition}

\begin{proposition}
If the families of bounded operators \,$\left\{\,T_{i}\,\right\}_{\, i \,\in\, I}$\, and \,$\left\{\,U_{j}\,\right\}_{ j \,\in\, J}$\, on \,$H$\, and \,$K$\, are the resolutions of the identity operator on \,$H$\, and \,$K$, then \,$\left\{\,T_{i} \,\otimes\, U_{j}\,\right\}_{i,\, j}$\, is a resolution of the identity operator on \,$H \,\otimes\, K$.
\end{proposition}

\begin{proof}
Since \,$\left\{\,T_{i}\,\right\}_{\, i \,\in\, I}\, \;\text{and}\; \left\{\,U_{j}\,\right\}_{ j \,\in\, J}$\, are the resolutions of the identity operator on \,$H\, \;\text{and}\; \,K$, respectively, we have
\[f \,=\, \sum\limits_{\,i \,\in\, I}\,T_{i}\,(\,f\,)\;\; \;\forall\; f \,\in\, H\; \;\text{and}\;\; \;g \,=\, \sum\limits_{\,j \,\in\, J}\,U_{j}\,(\,g\,)\;\; \;\forall\; g \,\in\, K.\]
Then, for all \,$f \,\otimes\, g \,\in\, H \,\otimes\, K$, we have
\[f \,\otimes\, g \,=\, \left(\,\sum\limits_{\,i \,\in\, I}\,T_{i}\,(\,f\,)\,\right) \,\otimes\, \left(\,\sum\limits_{\,j \,\in\, J}\,U_{j}\,(\,g\,)\,\right) \,=\, \sum\limits_{i,\, j}\,\left(\,T_{i} \,\otimes\, U_{j}\,\right)\,(\,f \,\otimes\, g\,).\]This completes the proof.  
\end{proof}

\begin{note}
Let \,$V$\, and \,$W$\, be fusion frames for \,$H$\, and \,$K$\, with their associated frame operators \,$S_{V}$\, and \,$S_{W}$, respectively.\,By reconstruction formula we can write
\[f \,=\, \sum\limits_{\,i \,\in\, I}\,v^{\,2}_{\,i}\,S^{\,-\, 1}_{V}\,P_{\,V_{\,i}}\,(\,f\,)\;\; \;\forall\; f \,\in\, H\; \;\text{and}\;\;g \,=\, \sum\limits_{\,j \,\in\, J}\,w^{\,2}_{\,j}\,S^{\,-\, 1}_{W}\,P_{\,W_{\,j}}\,(\,g\,)\; \;\forall\, g \,\in\, K.\]
Then it is easy to verify that 
\[f \,\otimes\, g \,=\, \sum\limits_{i,\, j}\,v^{\,2}_{\,i}\,w^{\,2}_{\,j}\,S^{\,-\, 1}_{V \,\otimes\, W}\,P_{V_{i} \,\otimes\, W_{j}}\,(\,f \,\otimes\, g\,)\; \;\forall\; f \,\otimes\, g \,\in\, H \,\otimes\, K.\]
This shows that the family of operators \,$\left\{\,v^{\,2}_{\,i}\,w^{\,2}_{\,j}\,S^{\,-\, 1}_{V \,\otimes\, W}\,P_{V_{i} \,\otimes\, W_{j}}\,\right\}_{i,\,j}$\, is resolution of the identity operator on \,$H \,\otimes\, K$.
\end{note}

\begin{theorem}
Let \,$V$\, and \,$W$\, be fusion frames for \,$H$\, and \,$K$\, with frame bounds \,$(\,A,\, B\,)$\, and \,$(\,C,\, D\,)$\, having their associated frame operators \,$S_{V}$\, and \,$S_{W}$, respectively.\;Then \,$\left\{\,v^{\,2}_{\,i}\,w^{\,2}_{\,j}\,\left(\,T_{i} \,\otimes\, U_{j}\,\right)\,\right\}_{i,\, j}$\, is a resolution of the identity operator on \,$H \,\otimes\, K$, where \,$T_{i} \,\otimes\, U_{j} \,=\, P_{\,V_{i} \,\otimes\, W_{j}}\,S^{\,-\, 1}_{\,V \,\otimes\, \,W}$\, for \,$i \,\in\, I$\, and \,$j \,\in\, J$.\;Furthermore, 
\[\dfrac{A\,C}{B^{\,2}\,D^{\,2}}\,a^{\,2}\,b^{\,2}\left\|\,f \otimes g\,\right\|^{\,2} \,\leq \sum\limits_{i,\, j}v^{\,2}_{\,i}\,w^{\,2}_{\,j}\left\|\,\left(\,T_{i} \,\otimes\, U_{j}\,\right)(\,f \otimes g\,)\,\right\|^{\,2} \,\leq \dfrac{B\,D}{A^{\,2}\,C^{\,2}}\,a^{\,2}\,b^{\,2}\left\|\,f \otimes g\,\right\|^{\,2},\]
for all \,$f \,\otimes\, g \,\in\, H \,\otimes\, K$, where \,$a$\, and \,$b$\, are constants with \,$a\,b \,=\, 1$.   
\end{theorem}

\begin{proof}
Since \,$S_{V}$\, and \,$S_{W}$\, are frame operator for  \,$V$\, and \,$W$, for all \,$f \,\in\, H,\, g \,\in\, K$,
\[f \,=\, \sum\limits_{\,i \,\in\, I}\, v^{\,2}_{\,i}\,P_{\,V_{i}}\,\left(\,S^{\,-\, 1}_{V}\,f\,\right)\; \;\text{and}\;\; \;g \,=\, \sum\limits_{\,j \,\in\, J}\, w^{\,2}_{\,j}\,P_{\,W_{j}}\,\left(\,S^{\,-\, 1}_{W}\,g\,\right).\]Now, for all \,$f \,\otimes\, g \,\in\, H \,\otimes\, K$, we have
\[f \,\otimes\, g \,=\, \left(\,\sum\limits_{\,i \,\in\, I}\, v^{\,2}_{\,i}\,P_{\,V_{i}}\,\left(\,S^{\,-\, 1}_{V}\,f\,\right)\,\right) \,\otimes\, \left(\,\sum\limits_{\,j \,\in\, J}\, w^{\,2}_{\,j}\,P_{\,W_{j}}\,\left(\,S^{\,-\, 1}_{W}\,g\,\right)\,\right)\]
\[=\, \sum\limits_{i,\, j}\,v^{\,2}_{\,i}\,w^{\,2}_{\,j}\,\left(\,P_{\,V_{i}}\,S^{\,-\, 1}_{V}\,(\,f\,) \,\otimes\, P_{\,W_{j}}\,S^{\,-\, 1}_{W}\,(\,g\,)\,\right)\hspace{1cm}\]
\[\hspace{.3cm}=\, \sum\limits_{i,\, j}\,v^{\,2}_{\,i}\,w^{\,2}_{\,j}\,\left(\,P_{\,V_{i}} \,\otimes\, P_{\,W_{j}}\,\right)\,\left(\,S^{\,-\, 1}_{V} \,\otimes\, S^{\,-\, 1}_{W}\,\right)\,\left(\,f \,\otimes\, g\,\right)\]
\[=\, \sum\limits_{i,\, j}\,v^{\,2}_{\,i}\,w^{\,2}_{\,j}\;P_{\,V_{i} \,\otimes\, W_{j}}\,S^{\,-\, 1}_{\,V \,\otimes\, \,W}\,\left(\,f \,\otimes\, g\,\right).\hspace{2cm}\]
This shows that \,$\left\{\,v^{\,2}_{\,i}\,w^{\,2}_{\,j}\,\left(\,T_{i} \,\otimes\, U_{j}\,\right)\,\right\}_{i,\, j}$\, is a resolution of the identity operator on \,$H \,\otimes\, K$, where 
\,$T_{i} \,\otimes\, U_{j} \,=\, P_{\,V_{i} \,\otimes\, W_{j}}\,S^{\,-\, 1}_{\,V \,\otimes\, \,W} \,=\, P_{\,V_{i}}\,S^{\,-\, 1}_{V} \,\otimes\, P_{\,W_{j}}\,S^{\,-\, 1}_{W}$.\;Now, by \,$(\,VI\,)$\, of the Theorem (\ref{th1.1}), there exist constants \,$a$\, and \,$b$\, with \,$a\,b \,=\, 1$\, such that
\[ T_{\,i}\,(\,f\,) \,=\, a\;P_{\,V_{i}}\,S^{\,-\, 1}_{V}\,(\,f\,)\;\, \forall\; f \,\in\, H,\;\; \;\text{and}\;\; \;U_{\,j}\,(\,g\,) \,=\, b\;P_{\,W_{j}}\,S^{\,-\, 1}_{W}\,(\,g\,)\; \;\forall\; g \,\in\, K.\]
Then, for all \,$f \,\otimes\, g \,\in\, H \,\otimes\, K$, we have
\[\sum\limits_{i,\, j}\,v^{\,2}_{\,i}\,w^{\,2}_{\,j}\,\left\|\,\left(\,T_{i} \,\otimes\, U_{j}\,\right)\,(\,f \,\otimes\, g\,)\,\right\|^{\,2} \,=\, \sum\limits_{i,\, j}\,v^{\,2}_{\,i}\,w^{\,2}_{\,j}\,\left\|\,T_{i}\,(\,f\,) \,\otimes\, U_{j}\,(\,g\,)\,\right\|^{\,2}\]
\[=\, \sum\limits_{i,\, j}\,v^{\,2}_{\,i}\,w^{\,2}_{\,j} \left\|\,T_{i}\,(\,f\,)\,\right\|_{1}^{\,2} \left\|\,U_{j}\,(\,g\,)\,\right\|_{2}^{\,2} \,=\, \left(\,\sum\limits_{\,i \,\in\, I}\,v^{\,2}_{\,i} \left\|\,T_{i}\,(\,f\,)\,\right\|_{1}^{\,2}\,\right)\,\left(\,\sum\limits_{\,j \,\in\, J}\,w^{\,2}_{\,j}\,\left\|\,U_{j}\,(\,g\,)\,\right\|_{2}^{\,2}\,\right)\]
\begin{equation}\label{eq3.1}
=\, \left(\,\sum\limits_{\,i \,\in\, I}\,v^{\,2}_{\,i}\,\left\|\,a\;P_{\,V_{i}}\,S^{\,-\, 1}_{V}\,(\,f\,)\,\right\|_{1}^{\,2}\,\,\right)\, \left(\,\sum\limits_{\,j \,\in\, J}\,w^{\,2}_{\,j}\,\left\|\,b\;P_{\,W_{j}}\,S^{\,-\, 1}_{W}\,(\,g\,)\,\right\|_{2}^{\,2}\,\,\right)\hspace{.8cm}
\end{equation}  
\[\leq\, B\;D\;a^{\,2}\;b^{\,2}\,\left\|\,S^{\,-\, 1}_{V}\,(\,f\,)\,\right\|_{1}^{\,2}\;\left\|\,S^{\,-\, 1}_{W}\,(\,g\,)\,\right\|_{2}^{\,2}\; \;[\;\text{since $V,\;W$ are fusion frames}\;]\hspace{.2cm}\]
\[\leq\, \dfrac{B\,D}{A^{\,2}\,C^{\,2}}\;\;a^{\,2}\;b^{\,2}\;\|\,f\,\|_{\,1}^{\,2}\;\|\,g\,\|_{\,2}^{\,2} \,=\, \dfrac{B\,D}{A^{\,2}\,C^{\,2}}\;\;a^{\,2}\;b^{\,2}\; \|\,f \,\otimes\, g\,\|^{\,2}.\hspace{3cm}\]
\[\left[\;\text{since $B^{\,-1}\,I_{H} \,\leq\, S^{\,-1}_{V} \,\leq\, A^{\,-1}\,I_{H}$\, and \,$D^{\,-1}\,I_{K} \,\leq\, S^{\,-1}_{W} \,\leq\, C^{\,-1}\,I_{K}$}\,\right].\]
On the other hand, using (\ref{eq3.1})
\[\sum\limits_{i,\, j}\,v^{\,2}_{\,i}\,w^{\,2}_{\,j}\,\left\|\,\left(\,T_{i} \,\otimes\, U_{j}\,\right)\,(\,f \,\otimes\, g\,)\,\right\|^{\,2} \,\geq\, A\;C\;a^{\,2}\;b^{\,2}\,\left\|\,S^{\,-\, 1}_{V}\,(\,f\,)\,\right\|_{\,1}^{\,2}\;\left\|\,S^{\,-\, 1}_{W}\,(\,g\,)\,\right\|_{\,2}^{\,2}\]
\[\hspace{3cm}\geq\, \dfrac{A\,C}{B^{\,2}\,D^{\,2}}\;a^{\,2}\;b^{\,2}\;\|\,f\,\|_{\,1}^{\,2}\;\|\,g\,\|_{\,2}^{\,2} \,=\, \dfrac{A\,C}{B^{\,2}\,D^{\,2}}\;\;a^{\,2}\;b^{\,2}\,\left\|\,f \,\otimes\, g\,\right\|^{\,2}.\]
This completes the proof.       
\end{proof}

\section{Alternative dual fusion frame in tensor product of Hilbert spaces}

\smallskip\hspace{.6 cm}In this section, an alternative dual of a fusion frame in \,$H \,\otimes\, K$\, is discussed.\\

\begin{theorem}
Let \,$V$\, and \,$W$\, be fusion frames for \,$H$\, and \,$K$\, with frame bounds \,$A,\, B$\, and \,$C,\, D$\, having their corresponding fusion frame operators \,$S_{V}$\, and \,$S_{W}$, respectively.\;Then the family \,$\Lambda \,=\, \, \left\{\,S^{\,-\, 1}_{V \,\otimes\, W}\,\left(\,V_{i} \,\otimes\, W_{j}\,\right),\, v_{\,i}\,w_{\,j}\,\right\}_{i,\, j}$\, is a fusion frame for \,$H \,\otimes\, K$.
\end{theorem}

\begin{proof}
By Theorem (\ref{thm1.1}), for all \,$f \,\in\, H$\, and \,$g \,\in\, K$, we have
\begin{equation}\label{eq4}
\dfrac{A\,\|\,f\,\|_{1}^{\,2}}{\left\|\,S_{V}\,\right\|^{\,2}\,\left\|\,S^{\,-\, 1}_{V}\,\right\|^{\,2}} \,\leq\, \sum\limits_{\,i \,\in\, I}\,v^{\,2}_{\,i}\,\left\|\,P_{\,S^{\,-\, 1}_{V}\,V_{\,i}}\,(\,f\,)\,\right\|_{1}^{\,2} \,\leq\, B\,\left\|\,S_{V}\,\right\|^{\,2}\,\left\|\,S^{\,-\, 1}_{V}\,\right\|^{\,2}\,\|\,f\,\|_{1}^{\,2},
\end{equation}

\begin{equation}\label{eq4.1}
\dfrac{C\,\|\,g\,\|_{2}^{\,2}}{\left\|\,S_{W}\,\right\|^{\,2}\,\left\|\,S^{\,-\, 1}_{W}\,\right\|^{\,2}}\leq\sum\limits_{\,j \,\in\, J}\,w^{\,2}_{\,j}\,\left\|\,P_{\,S^{\,-\, 1}_{W}\,W_{\,j}}\,(\,g\,)\,\right\|_{2}^{\,2}\leq D\,\left\|\,S_{W}\,\right\|^{\,2}\,\left\|\,S^{\,-\, 1}_{W}\,\right\|^{\,2}\,\|\,g\,\|_{2}^{\,2}
\end{equation}
Multiplying the inequalities (\ref{eq4}) and (\ref{eq4.1}) and using (\ref{eqn1.01}), we get
\[\dfrac{A\,C\,\|\,f \,\otimes\, g\,\|^{\,2}}{\left\|\,S_{V}\,\right\|^{\,2}\,\left\|\,S_{W}\,\right\|^{\,2}\,\left\|\,S^{\,-\, 1}_{V}\,\right\|^{\,2}\,\left\|\,S^{\,-\, 1}_{W}\,\right\|^{\,2}} \,\leq\, \sum\limits_{\,i,\,j}\,v^{\,2}_{\,i}\,w^{\,2}_{\,j}\,\left\|\,P_{\,S^{\,-\, 1}_{V}\,V_{\,i}}\,(\,f\,) \,\otimes\, P_{\,S^{\,-\, 1}_{W}\,W_{\,j}}\,(\,g\,)\,\right\|^{\,2}\]
\[\hspace{4.5cm} \,\leq\, B\,D\,\left\|\,S_{V}\,\right\|^{\,2}\,\left\|\,S_{W}\,\right\|^{\,2}\,\left\|\,S^{\,-\, 1}_{V}\,\right\|^{\,2}\,\left\|\,S^{\,-\, 1}_{W}\,\right\|^{\,2}\,\|\,f \,\otimes\, g\,\|^{\,2}.\]
Therefore, for each \,$f \,\otimes\, g \,\in\, H \,\otimes\, K$, we get
\[\Rightarrow\, \dfrac{A\,C\,\|\,f \,\otimes\, g\,\|^{\,2}}{\left\|\,S_{V} \,\otimes\, S_{W}\,\right\|^{\,2}\,\left\|\,S^{\,-\, 1}_{V} \,\otimes\, S^{\,-\, 1}_{W}\,\right\|^{\,2}} \,\leq\, \sum\limits_{\,i,\,j}\,v^{\,2}_{\,i}\,w^{\,2}_{\,j}\,\left\|\,P_{\,S^{\,-\, 1}_{V}\,V_{\,i} \,\otimes\, S^{\,-\, 1}_{W}\,W_{\,j}}\,(\,f \,\otimes\, g\,)\,\right\|^{\,2}\]
\[\hspace{4.5cm} \,\leq\, B\,D\,\left\|\,S_{V} \,\otimes\, S_{W}\,\right\|^{\,2}\,\left\|\,S^{\,-\, 1}_{V} \,\otimes\, S^{\,-\, 1}_{W}\,\right\|^{\,2}\,\|\,f \,\otimes\, g\,\|^{\,2}\] 
\[\Rightarrow\, \dfrac{A\,C\,\|\,f \,\otimes\, g\,\|^{\,2}}{\left\|\,S_{V \,\otimes\, W}\,\right\|^{\,2}\,\left\|\,S^{\,-\, 1}_{V \,\otimes\, W}\,\right\|^{\,2}} \,\leq\, \sum\limits_{\,i,\,j}\,v^{\,2}_{\,i}\,w^{\,2}_{\,j}\,\left\|\,P_{\,S^{\,-\, 1}_{V \,\otimes\, W}\,\left(\,V_{\,i} \,\otimes\, W_{\,j}\,\right)}\,(\,f \,\otimes\, g\,)\,\right\|^{\,2}\hspace{1.5cm}\]
\[\hspace{3.2cm} \,\leq\, B\,D\,\left\|\,S_{V \,\otimes\, W}\,\right\|^{\,2}\,\left\|\,S^{\,-\, 1}_{V \,\otimes\, W}\,\right\|^{\,2}\,\|\,f \,\otimes\, g\,\|^{\,2}.\]
This shows that \,$\Lambda$\, is a fusion frame for \,$H \,\otimes\, K$\, with bounds \,$\dfrac{A\,C}{\left\|\,S_{V \,\otimes\, W}\,\right\|^{\,2}\,\left\|\,S^{\,-\, 1}_{V \,\otimes\, W}\,\right\|^{\,2}}$\, and \,$B\,D\,\left\|\,S_{V \,\otimes\, W}\,\right\|^{\,2}\,\left\|\,S^{\,-\, 1}_{V \,\otimes\, W}\,\right\|^{\,2}$.\\ 
\end{proof}

\begin{definition}
Let \,$V \,\otimes\, W$\, be a fusion frame for \,$H \,\otimes\, K$\, and \,$S_{V \,\otimes\, W}$\, be the corresponding  fusion frame operator.\;Then a fusion Bessel sequence \,$V^{\,\prime} \,\otimes\, W^{\,\prime}$\, in \,$H \,\otimes\, K$\, is said to be an alternative dual of \,$V \,\otimes\, W$\, if for all \,$f \,\otimes\, g \,\in\, H \,\otimes\, K$,
\[f \,\otimes\, g \,=\, \sum\limits_{i,\, j}\,v_{\,i}\,w_{\,j}\,v^{\,\prime}_{\,i}\,w^{\,\prime}_{\,j}\;P_{\,V^{\,\prime}_{i} \,\otimes\, W^{\,\prime}_{j}}\;S^{\,-\, 1}_{V \,\otimes\, W}\;P_{\,V_{i} \,\otimes\, W_{j}}\,(\,f \,\otimes\, g\,).\]  
\end{definition}

\begin{note}
According to note (\ref{note1}), a reconstruction formula on \,$K$\, is also described by 
\[g \,=\, \sum\limits_{\,j \,\in\, J}\,w^{\,2}_{\,j}\,P_{\,S^{\,-\, 1}_{W}\,W_{\,j}}\,S^{\,-\, 1}_{W}\,P_{\,W_{\,j}}\,(\,g\,)\;\; \;\forall\; g \,\in\, K.\]
Therefore, for all \,$f \,\otimes\, g \,\in\, H \,\otimes\, K$, we have
\[f \,\otimes\, g \,=\, \left(\,\sum\limits_{\,i \,\in\, I}\,v^{\,2}_{\,i}\,P_{\,S^{\,-\, 1}_{V}\,V_{\,i}}\,S^{\,-\, 1}_{V}\,P_{\,V_{\,i}}\,(\,f\,)\,\right) \,\otimes\, \left(\,\sum\limits_{\,j \,\in\, J}\,w^{\,2}_{\,j}\,P_{\,S^{\,-\, 1}_{W}\,W_{\,j}}\,S^{\,-\, 1}_{W}\,P_{\,W_{\,j}}\,(\,g\,)\,\right)\]
\[\hspace{1cm}=\, \sum\limits_{i,\, j}\,v^{\,2}_{\,i}\,w^{\,2}_{\,j}\,\left(\,P_{\,S^{\,-\, 1}_{V}\,V_{\,i}} \,\otimes\, P_{\,S^{\,-\, 1}_{W}\,W_{\,j}}\,\right)\,\left(\,S^{\,-\, 1}_{V} \,\otimes\, S^{\,-\, 1}_{W}\,\right)\,\left(\,P_{V_{i}} \,\otimes\, P_{W_{j}}\,\right)\,(\,f \,\otimes\, g\,)\]
\[=\, \sum\limits_{i,\, j}\,v^{\,2}_{\,i}\,w^{\,2}_{\,j}\,P_{\,S^{\,-\, 1}_{V}\,V_{\,i} \,\otimes\, S^{\,-\, 1}_{W}\,W_{j}}\,S^{\,-\, 1}_{V \,\otimes\, W}\,P_{V_{i} \,\otimes\, W_{j}}\,(\,f \,\otimes\, g\,)\hspace{2.8cm}\]
\[=\, \sum\limits_{i,\, j}\,v^{\,2}_{\,i}\,w^{\,2}_{\,j}\,P_{\,S^{\,-\, 1}_{V \,\otimes\, W}\,\left(\,V_{\,i} \,\otimes\, W_{j}\,\right)}\,S^{\,-\, 1}_{V \,\otimes\, W}\,P_{V_{i} \,\otimes\, W_{j}}\,(\,f \,\otimes\, g\,).\hspace{2.8cm}\]
Thus, we see that the canonical dual frame \,$\left\{\,S^{\,-\, 1}_{V \,\otimes\, W}\,\left(\,V_{i} \,\otimes\, W_{j}\,\right),\, v_{\,i}\,w_{\,j}\,\right\}_{i,\, j}$\, is an alternative dual fusion frame for \,$H \,\otimes\, K$.\\
\end{note}
   
\begin{theorem}\label{th3}
Let \,$V$\, and \,$W$\, be fusion frames for \,$H$\, and \,$K$\, with their alternative dual \,$V^{\,\prime} \,=\, \left\{\,\left(\,V^{\,\prime}_{i},\, v^{\,\prime}_{\,i}\,\right)\,\right\}_{\, i \,\in\, I}$\, and \,$W^{\,\prime} \,=\, \left\{\,\left(\,W^{\,\prime}_{j},\, w^{\,\prime}_{\,j}\,\right)\,\right\}_{ j \,\in\, J}$\,, respectively.\;Then \,$V^{\,\prime} \,\otimes\, W^{\,\prime}$\, is an alternative dual of the fusion frame \,$V \,\otimes\, W$\, for \,$H \,\otimes\, K$. 
\end{theorem}

\begin{proof}
By Theorem (\ref{thm2}), \,$V \,\otimes\, W$\, is a fusion frame for \,$H \,\otimes\, K$\, and \,$V^{\,\prime} \,\otimes\, W^{\,\prime}$\, is a fusion Bessel sequence in \,$H \,\otimes\, K$.\;Since \,$V^{\,\prime}$\, and \,$W^{\,\prime}$\, are alternative dual sequences of \,$V$\, and \,$W$, for all \,$f \,\in\, H\, \,\text{and}\, \,g \,\in\, K$, 
\[f \,=\, \sum\limits_{\,i \,\in\, I}\,v_{\,i}\,v^{\,\prime}_{\,i}\,P_{\,V^{\,\prime}_{i}}\;S^{\,-\, 1}_{V}\;P_{\,V_{i}}\,(\,f\,)\; \;\text{and}\; \;g \,=\, \sum\limits_{\,j \,\in\, J}\,w_{\,j}\,w^{\,\prime}_{\,j}\,P_{\,W^{\,\prime}_{j}}\;S^{\,-\, 1}_{\,W}\;P_{\,W_{j}}\,(\,g\,).\]Then, for all \,$f \,\otimes\, g \,\in\, H \,\otimes\, K$, using the Theorem (\ref{th1.1}), we get
\[f \,\otimes\, g \,=\, \left(\,\sum\limits_{\,i \,\in\, I}\,v_{\,i}\,v^{\,\prime}_{\,i}\,P_{\,V^{\,\prime}_{i}}\;S^{\,-\, 1}_{V}\;P_{\,V_{i}}\,(\,f\,)\,\right) \,\otimes\, \left(\,\sum\limits_{\,j \,\in\, J}\,w_{\,j}\,w^{\,\prime}_{\,j}\,P_{\,W^{\,\prime}_{j}}\;S^{\,-\, 1}_{W}\;P_{\,W_{j}}\,(\,g\,)\,\right)\]
\[=\, \sum\limits_{i,\, j}\,v_{\,i}\,w_{\,j}\,v^{\,\prime}_{\,i}\,w^{\,\prime}_{\,j}\,\left(\,P_{\,V^{\,\prime}_{i}} \,\otimes\, P_{\,W^{\,\prime}_{j}}\,\right)\,\left(\,S^{\,-\, 1}_{V} \,\otimes\, S^{\,-\, 1}_{W}\,\right)\,\left(\,P_{\,V_{i}} \,\otimes\, P_{\,W_{j}}\,\right)\,(\,f \,\otimes\, g\,)\]
\[\,=\, \sum\limits_{i,\, j}\,v_{\,i}\,w_{\,j}\,v^{\,\prime}_{\,i}\,w^{\,\prime}_{\,j}\;P_{\,V^{\,\prime}_{i} \,\otimes\, W^{\,\prime}_{j}}\;S^{\,-\, 1}_{V \,\otimes\, W}\;P_{\,V_{i} \,\otimes\, W_{j}}\,(\,f \,\otimes\, g\,).\hspace{3.7cm}\]
This completes the proof\,. 
\end{proof}

\begin{theorem}
Let \,$V$\, and \,$W$\, be fusion frames for \,$H$\, and \,$K$\, with bounds \,$(\,B_{\,1},\,D_{\,1}\,)$\, and \,$(\,B_{\,2},\,D_{\,2}\,)$\, having their alternative duals \,$V^{\,\prime} \,=\, \left\{\,\left(\,V^{\,\prime}_{i},\, v^{\,\prime}_{\,i}\,\right)\,\right\}_{\, i \,\in\, I}$\, and \,$W^{\,\prime} \,=\, \left\{\,\left(\,W^{\,\prime}_{j},\, w^{\,\prime}_{\,j}\,\right)\,\right\}_{ j \,\in\, J}$\,, respectively.\;Then \,$V^{\,\prime} \,\otimes\, W^{\,\prime}$\, is a fusion frame for \,$H \,\otimes\, K$.
\end{theorem}

\begin{proof}
Since \,$V^{\,\prime}$\, and \,$W^{\,\prime}$\, are fusion Bessel sequences in \,$H$\, and \,$K$, respectively, by Theorem (\ref{thm2}), \,$V^{\,\prime} \,\otimes\, W^{\,\prime}$\, is a fusion Bessel sequence in \,$H \,\otimes\, K$.\;Also, since \,$V^{\,\prime}$\, and \,$W^{\,\prime}$\, are alternative dual sequences of \,$V$\, and \,$W$, respectively, by Theorem (\ref{th3}), \,$V^{\,\prime} \,\otimes\, W^{\,\prime}$\, is an alternative dual of the fusion frame \,$V \,\otimes\, W$\, for \,$H \,\otimes\, K$.  Now, for each \,$f \,\otimes\, g \,\in\, H \,\otimes\, K$, we have \;$\left\|\,f \,\otimes\, g\,\right\|^{\,2} \,=\, \left<\,f \,\otimes\, g \,,\, f \,\otimes\, g\,\right>$
\[ \,=\, \left<\,\sum\limits_{i,\, j}\,v_{\,i}\,w_{\,j}\,v^{\,\prime}_{\,i}\,w^{\,\prime}_{\,j}\;P_{\,V^{\,\prime}_{i} \,\otimes\, W^{\,\prime}_{j}}\;S^{\,-\, 1}_{V \,\otimes\, W}\;P_{\,V_{i} \,\otimes\, W_{j}}\,(\,f \,\otimes\, g\,) \;,\; f \,\otimes\, g\,\right>\hspace{3.5cm}\]
\[=\, \sum\limits_{i,\, j}\,v_{\,i}\,w_{\,j}\,v^{\,\prime}_{\,i}\,w^{\,\prime}_{\,j}\,\left<\,S^{\,-\, 1}_{V \,\otimes\, W}\;P_{\,V_{i} \,\otimes\, W_{j}}\,(\,f \,\otimes\, g\,) \;,\; P_{\,V^{\,\prime}_{i} \,\otimes\, W^{\,\prime}_{j}}\,(\,f \,\otimes\, g\,)\,\right>\hspace{3.2cm}\]
\[=\, \sum\limits_{i,\, j}\,v_{\,i}\,w_{\,j}\,v^{\,\prime}_{\,i}\,w^{\,\prime}_{\,j}\,\left<\,S^{\,-\, 1}_{V}\;P_{\,V_{i}}\,(\,f\,) \,\otimes\, S^{\,-\, 1}_{W}\;P_{\,W_{j}}\,(\,g\,) \;,\; P_{\,V^{\,\prime}_{i}}\,(\,f\,) \,\otimes\, P_{\,W^{\,\prime}_{j}}\,(\,g\,)\,\right>\hspace{1.3cm}\]
\[=\, \left(\,\sum\limits_{\,i \,\in\, I}\,v_{\,i}\,v^{\,\prime}_{\,i}\,\left<\,S^{\,-\, 1}_{V}\;P_{\,V_{i}}\,(\,f\,) \,,\, P_{\,V^{\,\prime}_{i}}\,(\,f\,)\,\right>_{1}\,\right)\,\left(\,\sum\limits_{\,j \,\in\, J}\,w_{\,j}\,w^{\,\prime}_{\,j}\,\left<\,S^{\,-\, 1}_{W}\;P_{\,W_{j}}\,(\,g\,) \,,\, P_{\,W^{\,\prime}_{j}}\,(\,g\,)\,\right>_{2}\,\right)\]
\[\leq\, \left(\,\sum\limits_{\,i \,\in\, I}\,v_{i}\,v^{\,\prime}_{i}\,\left\|\,S^{\,-\, 1}_{V}\;P_{\,V_{i}}\,(\,f\,)\,\right\|_{1}\,\left\|\,P_{\,V^{\,\prime}_{i}}\,(\,f\,)\,\right\|_{1}\,\right)\,\left(\,\sum\limits_{\,j \,\in\, J}\,w_{j}\,w^{\,\prime}_{j}\,\left\|\,S^{\,-\, 1}_{W}\;P_{\,W_{j}}\,(\,g\,)\,\right\|_{2}\,\left\|\,P_{\,W^{\,\prime}_{j}}\,(\,g\,)\,\right\|_{2}\,\right)\]
\[\leq\, \left(\,\sum\limits_{\,i \,\in\, I}\,v^{\,2}_{i}\,\left\|\,S^{\,-\, 1}_{V}\;P_{\,V_{i}}\,(\,f\,)\,\right\|_{1}^{\,2}\,\right)^{\,\dfrac{1}{2}}\,\left(\,\sum\limits_{\,i \,\in\, I}\,(\,v^{\,\prime}_{i}\,)^{\,2}\,\left\|\,P_{\,V^{\,\prime}_{i}}\,(\,f\,)\,\right\|_{1}^{\,2}\,\right)^{\,\dfrac{1}{2}}\left(\,\sum\limits_{\,j \,\in\, J}\,w^{\,2}_{j}\,\left\|\,S^{\,-\, 1}_{W}\;P_{\,W_{j}}\,(\,g\,)\,\right\|_{2}^{\,2}\,\right)^{\,\dfrac{1}{2}}\,\times\hspace{2.7cm}\]
\[\left(\,\sum\limits_{\,j \,\in\, J}\,(\,w^{\,\prime}_{\,j}\,)^{\,2}\,\left\|\,P_{\,W^{\,\prime}_{j}}\,(\,g\,)\,\right\|_{\,2}^{\,2}\,\right)^{\,\dfrac{1}{2}}\; \;[\;\text{by C-S inequality}\;]\]
\[\leq\, \sqrt{D_{1}\,D_{2}}\,\left\|\,S^{\,-\, 1}_{V}\,\right\|\,\left\|\,S^{\,-\, 1}_{W}\,\right\|\,\|\,f\,\|_{1}\,\|\,g\,\|_{2}\,\left(\,\sum\limits_{\,i \,\in\, I}\,(\,v^{\,\prime}_{i}\,)^{\,2}\,\left\|\,P_{\,V^{\,\prime}_{i}}\,(\,f\,)\,\right\|_{1}^{\,2}\,\right)^{\,\dfrac{1}{2}}\,\left(\,\sum\limits_{\,j \,\in\, J}\,(\,w^{\,\prime}_{j}\,)^{\,2}\,\left\|\,P_{\,W^{\,\prime}_{j}}\,(\,g\,)\,\right\|_{2}^{\,2}\,\right)^{\,\dfrac{1}{2}}\]
\[\;[\;\text{since $V,\;W$ are fusion frames}\;]\]
\[=\, \sqrt{D_{\,1}\,D_{\,2}}\;\left\|\,S^{\,-\, 1}_{V \,\otimes\, W}\,\right\|\,\left\|\,f \,\otimes\, g\,\right\|\,\left(\,\sum\limits_{i,\, j}\,(\,v^{\,\prime}_{\,i}\,)^{\,2}\,(\,w^{\,\prime}_{\,j}\,)^{\,2}\,\left\|\,P_{\,V^{\,\prime}_{i} \,\otimes\, W^{\,\prime}_{j}}\,(\,f \,\otimes\, g\,)\,\right\|^{\,2}\,\right)^{\,\dfrac{1}{2}}\hspace{.5cm}\]
\[\Rightarrow\, \dfrac{1}{D_{\,1}\,D_{\,2}\,\left\|\,S^{\,-\, 1}_{V \,\otimes\, W}\,\right\|^{\,2}}\,\left\|\,f \,\otimes\, g\,\right\|^{\,2} \,\leq\, \sum\limits_{i,\, j}\,(\,v^{\,\prime}_{\,i}\,)^{\,2}\,(\,w^{\,\prime}_{\,j}\,)^{\,2}\,\left\|\,P_{\,V^{\,\prime}_{i} \,\otimes\, W^{\,\prime}_{j}}\,(\,f \,\otimes\, g\,)\,\right\|^{\,2}.\hspace{1cm}\]
This completes the proof. 
\end{proof}

\end{document}